\documentclass[11pt, a4paper,pdftex]{amsart} 
\author{Jeffrey Giansiracusa} 

\author{Noah  Giansiracusa} 

\usepackage{mathptmx}
\DeclareSymbolFont{cmlargesymbols}{OMX}{cmex}{m}{n}
\DeclareMathSymbol{\mycoprod}{\mathop}{cmlargesymbols}{"60}

\usepackage[mathcal]{eucal}
\DeclareFontFamily{OT1}{pzc}{}
\DeclareFontShape{OT1}{pzc}{m}{it}{<-> s * [1.10] pzcmi7t}{}
\DeclareMathAlphabet{\mathpzc}{OT1}{pzc}{m}{it}

  \usepackage[pdftex]{graphicx}

\usepackage[protrusion=true,expansion=true]{microtype}

\usepackage{amssymb}
\usepackage{mathrsfs}  
\usepackage{latexsym}
\usepackage{amsmath}
\usepackage{color}
\usepackage{hyperref}

\usepackage[cmtip,arrow]{xy}
\usepackage{pb-diagram,pb-xy}
\usepackage[vmargin=3cm, hmargin=3cm]{geometry}
\numberwithin{equation}{subsection}

\newtheorem{theorem}{Theorem}[subsection]  
\newtheorem*{theorem*}{Theorem}  
\newtheorem{lemma}[theorem]{Lemma} 
\newtheorem{proposition}[theorem]{Proposition}
\newtheorem{corollary}[theorem]{Corollary}

\theoremstyle{remark} 
\newtheorem{definition}[theorem]{Definition}
\newtheorem{question}[theorem]{Question}
\newtheorem{remark}[theorem]{Remark}
\newtheorem{example}[theorem]{Example}

\newcommand{\R}{\mathbb{R}}

\newcommand{\T}{\mathbb{T}}
\newcommand{\PP}{\mathbb{P}}

\newcommand{\bend}{\mathpzc{B}}

\newcommand{\supp}{\mathrm{supp}}

\newcommand{\Sym}{\mathrm{Sym}}
\newcommand{\Hom}{\mathrm{Hom}}

\newcommand{\tropker}{\mathrm{tropker}}
\newcommand{\ext}{{\bigwedge}}

\definecolor{purple}{rgb}{0.5,0,0.5}
\definecolor{brown}{rgb}{0.5,0.3,0.1}

\DeclareMathOperator{\im}{im} 

\mathchardef\mhyphen="2D

\usepackage{lipsum}

\newcommand\blfootnote[1]{%
  \begingroup
  \renewcommand\thefootnote{}\footnote{#1}%
  \addtocounter{footnote}{-1}%
  \endgroup
}

\newcommand{\Gr}{\operatorname{Gr}}

\title{A Grassmann algebra for matroids} 

\begin{document}
\begin{abstract}
 We introduce an idempotent analogue of the exterior algebra for which the theory of tropical linear spaces (and valuated matroids) can be seen in close analogy with the classical Grassmann algebra formalism for linear spaces.  The top wedge power of a tropical linear space is its Pl\"ucker vector, which we view as a tensor, and a tropical linear space is recovered from its Pl\"ucker vector as the kernel of the corresponding wedge multiplication map.  We prove that an arbitrary $d$-tensor satisfies the tropical Pl\"ucker relations (valuated exchange axiom) if and only if the $d^\text{th}$ wedge power of the kernel of wedge-multiplication is free of rank one.  This provides a new cryptomorphism for valuated matroids, including ordinary matroids as a special case.  
\end{abstract}
\maketitle



\section{Introduction}

Matroids, which provide a combinatorial abstraction of linear dependence over a field, have a deep and fascinating history \cite{Oxley,White1,White2,Katz-matroid}.  Valuated matroids are an enrichment where each basis in the matroid is weighted by an element of a fixed totally ordered abelian group and the usual exchange axiom is replaced by a valuated exchange axiom taking into account these weights \cite{Dress-Wenzel}.  Just as for matroids, there are many equivalent definitions of a valuated matroid, e.g., \cite{Murota-valuated, Murota-circuits, Dress-Wenzel-greedy}.

Associated to every matroid is a polyhedral complex known as the Bergman fan \cite{Bergman,Feichtner-Sturmfels} that uniquely determines the matroid.  Murota-Tamura \cite{Murota-circuits} and Speyer \cite{Speyer} extended this to the valuated case. The resulting polyhedral complexes are called tropical linear spaces since they generalize the tropicalizations of linear subspaces over a valued field and their elements are called vectors. As with the non-valuated case, such a polyhedral complex uniquely determines its valuated matroid.  Thus valuated matroids and tropical linear spaces are two sides of the same coin, one combinatorial and one geometric.  Many concepts in the theory of tropical linear spaces are inspired by their classical counterparts through `tropicalization' in a colloquial sense, yet the tools of the trade and proof techniques rest primarily in the framework of matroid polytopes.   

In this paper we introduce a tropical variant of the exterior algebra of a vector space that makes precise many of these analogies between classical and tropical linear spaces, and which provides a setting in which to study the latter using familiar ideas from (multi)linear algebra.  We derive a novel characterization of valuated matroids in terms of top wedge powers, revealing a striking parallel to vector spaces.  All of our results specialize to the case of ordinary matroids and so reinforce the connections to linear algebra/geometry at the heart of matroid theory.

Exterior algebras have appeared in the context of matroids previously, but playing a quite different role than ours.  The Orlik-Solomon algebra \cite{Orlik-Solomon} and the Whitney algebra \cite{Crapo-Schmitt} are both algebraic invariants of matroids, meaning they associate to each matroid an algebra encoding certain salient features of the matroid but not determining the matroid itself.  Our exterior algebra is associated to a ground set and is used to study all matroids on that ground set simultaneously, inspired by the Pl\"ucker embedding of the Grassmannian.  This is more similar in spirit to Dress and Wenzel's paper \cite{Dress-Wenzel-plucker}, though our constructions and results are nonetheless quite distinct from theirs, perhaps largely due to the emergence of tropical geometry since their paper was published.

\subsection{Summary of results}

Let $S$ be an idempotent semifield.  That is, $S$ satisfies all the axioms of a field except for the existence of additive inverses, and moreover $s+s = s$ for all $s\in S$.  For instance, given a totally ordered abelian group we obtain an idempotent semifield by letting the multiplication be the group operation and addition be maximum with respect to the ordering.  The most important examples, for us, obtained this way are the tropical numbers $\T = (\mathbb{R}\cup\{-\infty\},\max,+)$ and its Boolean subfield $\mathbb{B} = \{0,-\infty\}$.  However, to avoid confusion we shall always write the additive and multiplicative identities as $0$ and $1$, respectively.

Given a free $S$-module $V$ of rank $n < \infty$, we define the \emph{tropical Grassmann algebra} $\ext V$ to be the quotient of the symmetric algebra on $V$ by the relations identifying the squares $e_i^2$ of the elements of a basis $\{e_1,\ldots,e_n\}$ with zero.  Unique to the idempotent setting is that this does not depend on a choice of basis.  Despite the commutativity in the construction, many of the familiar properties of the exterior algebra on a vector space carry over.  For instance, this idempotent algebra is graded and the $d^{\text{th}}$ component $\ext^d V$ is free of rank $\binom{n}{d}$ with basis the elements $e_I := e_{i_1}\wedge\cdots\wedge e_{i_d}$ for each $I=\{i_1,\ldots,i_d\} \subset [n]$; given a collection of vectors $v_1,\ldots,v_d\in V$, the coefficients of $v_1\wedge\cdots\wedge v_d$ in this basis are the minors (in the tropical sense of permanents) of the matrix of coefficients of the $v_i$. More generally, if $V\to Q$ is a surjection of $S$-modules then we define the tropical Grassmann algebra $\ext Q$ to be the pushout of $\Sym~Q \leftarrow \Sym~V \to \ext V$.

A rank $d$ valuated matroid is a map $f : \binom{[n]}{d} \rightarrow S$ satisfying the valuated exchange axioms; we can now view this as a tensor $w = \sum f(I) e_I\in \ext^d V$, and the valuated exchange axiom becomes an algebraic (in fact, tropical, in the sense of \cite{GG1}) condition on the tensors in this idempotent module.  Speyer calls valuated matroids tropical Pl\"ucker vectors, since they play the role for tropical linear spaces that the usual Pl\"ucker vectors of maximals minors do for linear spaces over a field \cite{Speyer,Speyer-Sturmfels}.  The following results, though quite straightforward to prove, reinforce this analogy and illustrate the conceptual clarity in working with the tropical Grassmann algebra.  Let $V = \bigoplus_{i=1}^n Se_i$ as above, and let $V^\vee = \Hom(V,S)$ be the linear dual.
\begin{enumerate}
\item The tropical linear space $L_w \subset V$ associated to a tropical Pl\"ucker vector $w\in\ext^d V$ is the tropical kernel (see Definition \ref{def:tropker}) of the wedge multiplication map $-\wedge w : V \rightarrow \ext^{d+1} V$.
\item If $w_j\in \ext^{d_j}V$ are tropical Pl\"ucker vectors, $j=1,2$, then $w_1\wedge w_2 \in \ext^{d_1+d_2}V$ is the tropical Pl\"ucker vector of the stable sum (the operation dual to stable intersection in tropical geometry) of $L_{w_1}$ and $L_{w_2}$, when this is defined (cf., \S\ref{sec:stablesum}).
\item The rank $d$ Stiefel tropical linear spaces \cite{Fink-Rincon} correspond precisely to the totally decomposable tensors in $\PP\left(\ext^d V\right)$.
\item The valuated matroid elongation of $w\in\ext^d V$ is given by $w \wedge \left(\sum_{i=1}^n e_i\right) \in \ext^{d+1}V$.
\item Fixing an isomorphism $V\cong V^\vee$, there is a Hodge star operator $\star: \ext^d V \stackrel{\cong}{\to} \ext^{n-d} V$ that sends the tropical Pl\"ucker vector of a tropical linear space $L\subset V$ to that of its orthogonal dual $L^\perp \subset V$.
\end{enumerate}
 
Over a field $k$, the tensors in $\ext^d k^n$ which are Pl\"ucker vectors (i.e., lie on the Grassmannian $\Gr(d,n)$ in its Pl\"ucker embedding) are those for which the corresponding wedge multiplication map has $d$-dimensional kernel, which in turn is equivalent to this kernel having 1-dimensional $d^{\text{th}}$ wedge power.  In the tropical setting, to any nonzero tensor $w \in \ext^d V$ we can associate the tropical kernel $L_w\subset V$ of the corresponding wedge map $-\wedge w : V \to \ext^{d+1} V$, and this tropical kernel is a tropical linear space if and only if $w$ is a tropical Pl\"ucker vector.  Our main result, whose proof is significantly more involved than that of the preceding results, is that the $d^{\text{th}}$ tropical wedge power similarly characterizes when $w$ is a tropical Pl\"ucker vector, and hence gives a new cryptomoprhism for (valuated) matroids.  Given an arbitrary vector $w$, we define a quotient $Q_w$ of $V^\vee$ by imposing the bend relations of $(-\wedge w)$; note that $Q_w^\vee = L_w$, though it is not always the case that $L_w^\vee = Q_w$. 
Our main result is:

\begin{theorem*} 
Let $w\in\ext^d V$ be a nonzero element with associated quotient module $Q_w$.  Then $w$ is a tropical Pl\"ucker vector if and only if $\ext^d Q_w$ is free of rank one.  Moreover, in this case the line $\left( \ext^d Q_w\right) ^\vee \in \PP(\ext^d V)$ is the tropical Pl\"ucker vector of the tropical linear space $L_w=\tropker(-\wedge w)$ associated with $w$.
\end{theorem*}

\subsection*{Acknowledgements}
JG was supported by EPSRC grant EP/I003908/2, and he thanks Johns Hopkins University, where much of this work was carried out.  NG was supported by NSA grant H98230-16-1-0015.  We thank Matt Baker, Alex Fink, Diane Maclagan, Andrew Macpherson, Josh Mundinger, Sam Payne, Felipe Rinc\'on, and an anonymous referee for useful discussions and comments.  
 

\section{Preliminaries}

In this section we establish some notation to be used throughout the paper, and we recall some relevant concepts from the theory of semirings \cite{Golan} and from the scheme-theoretic perspective of tropical geometry \cite{GG1}.  We also discuss the useful concept of the ``tropical kernel'' of a linear map in the tropical setting.

\subsection{Semirings, their modules, and quotients}\label{sec:pre-cong}
Let $S$ denote an idempotent semifield; as mentioned in the introduction, being a \emph{semifield} means $S$ satisfies all the axioms of a field except for the existence of additive inverses, and the \emph{idempotent} property means $s+s=s$ for all $s\in S$.  We will write $0_S$ and $1_S$ for the additive and multiplicative units of $S$, respectively, or simply $0$ and $1$ if there is no chance of confusion.  One of the main examples of an idempotent semifield is the tropical real numbers $\T=(\R\cup \{-\infty\},\max,+)$, and in this semifield $0_\T = -\infty$ and $1_\T = 0$.

For any $S$-module $M$ and integer $d \ge 0$, we let $\Sym^d M$ denote the $d^{\text{th}}$ symmetric power of $M$ and $\Sym~M = \bigoplus_{d \ge 0}\Sym^d M$ the symmetric algebra; these are defined exactly as they are for modules over a ring.

A \emph{congruence} on an $S$-module $M$ is an equivalence relation which is a submodule of $M\times M$; the quotient by this equivalence relation is then an $S$-module, and in fact quotients of $M$ are in bijection with congruences on $M$.  Submodules of $M$ only define a special class of congruence in which certain elements are identified with zero (of course, over a ring every congruence is of this form, since $m\sim m' \Leftrightarrow m-m'\sim 0$).  Similarly, an algebra congruence on an $S$-algebra $A$ is an equivalence relation that is an $S$-subalgebra of $A$, and again these are in bijection with quotient $S$-algebras of $A$.

Given a homomorphism $\varphi : M \rightarrow N$ of $S$-modules, or $S$-algebras, we shall use the term \emph{congruence kernel} to denote the congruence $M\times_N M = \{(m,m')\in M\times M~|~\varphi(m)=\varphi(m')\}$ and reserve the unadorned term \emph{kernel} for the pre-image of zero, which is an ideal.

For more on these topics, one can consult \cite{Golan} in general or \cite[\S2.4]{GG1} for specifics on the concepts mentioned above.

\subsection{Free modules, notation, and conventions} 

Throughout this paper we will fix an idempotent semifield $S$ and a free $S$-module $V$ of finite rank $n$. Note that, unless explicitly stated otherwise, we do not require that $S$ is totally ordered.  For any $S$-module $M$, we denote the \emph{linear dual} by \[M^\vee = \Hom_S(M,S).  \] We will write $\langle - , -\rangle$ for the canonical bilinear pairing $V\times V^\vee \to S$.  Note that there is a canonical isomorphism $V\cong V^{\vee\vee}$.  We will fix a basis $\{e_i\}$ for $V$ and let $\{x_i\}$ denote the dual basis for $V^\vee$; it is characterized by \[ \langle e_i , x_j\rangle = \begin{cases} 1_S & \mbox{if } i=j, \\ 0_S & \mbox{if } i\neq j.\end{cases} \]

\begin{remark}
Although we choose a basis for $V$ and define constructions in terms of this choice, the constructions in this paper are actually independent of the choice.  Indeed, this is implied by the following result, whose proof was kindly communicated to us by Josh Mundinger (in this paper we only use the result for semifields, but we state here Mundinger's more general result).
\end{remark}

\begin{proposition}\label{prop:GL}
	Let $R$ be an idempotent semiring with no zero divisors.
    Then any basis of $R^n$ is unique up to multiplication by units and permutations,
    i.e., $GL(n,R) \cong \Sigma_n \ltimes (R^\times)^n$.
\end{proposition}
\begin{proof}
	Suppose $e_1, \ldots, e_n$ and $f_1, \ldots, f_n$ are two bases for $R^n$.
    Then we may write 
    \begin{equation*}
    	e_i = \sum_{j} a_i^j f_j,	\qquad 
    	f_i = \sum_j b_i^j e_j.
    \end{equation*}
    Thus $e_i = \sum_{j,k} a_i^j b_j^k e_k$,
i.e., $\sum_j a_i^j b_j^i = 1$ and for $k \neq i$ we have 
    \begin{equation} \label{matrix-multiplication}
  		\sum_j a_i^j b_j^k = 0.
    \end{equation}
    Assume by permuting our bases that $a_n^nb_n^n \neq 0$.
    Thus, both $a_n^n$ and $b_n^n$ are non-zero.
    Idempotency implies that each term in the sum in \eqref{matrix-multiplication} must be zero, so 
    $a_n^n b_n^k = 0$ for all $k \neq n$.
    Since $R$ has no zero divisors, we then have $b_n^k = 0$ for all $k \neq n$.
    In particular, $f_n = b_n^n e_n$.  By the same argument with the roles of the bases reversed, we have $a_n^k = 0$ for all $k \ne n$ and $e_n = a_n^n f_n$.
	Moreover, we deduce that $e_n = a_n^nb_n^n e_n$, and so since $e_n$ is an element of a basis we must have 	$a_n^nb_n^n = 1$, hence $a^n_n$ and $b^n_n$ are both units.
    Thus, $\langle e_1, \ldots, e_{n-1} \rangle = \langle f_1, \ldots, f_{n-1}\rangle$
    and so the result follows by induction on $n$.   
\end{proof}

We use the notation $[n] := \{1,\ldots,n\}$, and for $I\subset [n]$ we write $I+j := I\cup\{j\}$ and $I-j := I\smallsetminus\{j\}$.

\subsection{Tropical hyperplanes and the bend relations}\label{sec:pre-bend}

Let $f = \sum f_ix_i \in V^\vee$ be a linear form.  When $S=\T$, the \emph{tropical hyperplane} of $f$ has been defined as the locus in $\T^n$ where the maximum in $f$ is attained at least twice or $f$ attains the value $-\infty = 0_\T \in \T$ \cite{Sturmfels-first-steps,Mikhalkin-ICM}.  Note that this definition makes sense not just for $\T$ but for any totally ordered idempotent semifield.

In \cite{GG1} we proposed that tropical hyperplanes are the solution sets to systems of $S$-linear equations canonically associated with their defining linear forms.  These equations are called the \emph{bend relations} and they exist not just for $S$ a totally ordered idempotent semifield, but any idempotent semiring.

\begin{definition}
The \emph{bend relations} of a linear form $f = \sum f_ix_i \in V^\vee$ are the relations
\[\left\{f\sim \sum_{i\ne j}f_ix_i\right\}_{j=1}^n.\]  We write $\bend(f)$ for the $S$-module congruence generated by these relations.  
The \emph{tropical hyperplane} $f^\perp \subset V$ is the linear dual of the quotient:
\[f^\perp = (V^\vee/\bend(f))^\vee \subset V^{\vee\vee} \cong V.\]  If $L \subset V^\vee$ is a submodule, then $\bend(L)$ denotes the congruence generated by $\bend(f)$ for all $f\in L$.  If $M$ is an arbitrary module and $\varphi :V \to M$ is a linear map then we write $\bend(\varphi)$ for the congruence $\bend(L)$ where $L\subset V^\vee$ is the image of the dual map $\varphi^\vee: M^\vee \to V^\vee$. 
\end{definition}

\begin{remark}\label{rem:bend}
Here are some comments on these definitions.
\begin{enumerate}
\item The condition that $v = \sum v_i e_i \in V$ satisfies the bend relations of a linear form $f$ says that any single term in the summation $f(v) = \sum f_i v_i$ can be omitted without changing the value of the sum.  When $S=\T$, or any other totally ordered semifield, this
is equivalent to requiring that the maximum element of the sequence $\{f_i v_i\}_{i=1}^n$ is attained
at least twice (or there is only a single term and it vanishes), for if the maximum were attained
exactly once then deleting the unique maximal term would strictly decrease the value of the sum.   See  \cite[Proposition 5.1.6]{GG1}.
\item In \cite{GG1} this notion of tropical hyperplane was termed a ``set-theoretic bend locus'' to avoid confusion with two distinct scheme-theoretic notions of a hypersurface in the tropical setting.  Since in this paper we primarily discuss modules and linear forms rather than polynomials and schemes, we feel free to use the term tropical hyperplane here without confusion.
\item By transitivity and idempotency, the congruence $\bend(f)$ is also generated by
\[\{\sum_{i \ne j} f_ix_i \sim \sum_{i \ne k} f_ix_i\}_{j,k\in [n]}\]
\item The notation $f^\perp$ is borrowed from a more general notion of tropical orthogonal dual discussed later (see \S\ref{sec:duals}).
\item For a submodule $L\subset V^\vee$, we have \[(V^\vee/\bend(L))^\vee = \bigcap_{f\in L} f^\perp \subset V.\]
\end{enumerate}
\end{remark}

\subsection{Tropical kernels}\label{sec:pre-ker}

Tropical kernels first appeared in \cite{Sturmfels-first-steps} as the tropical analogue of the kernel of a linear map between free $\T$-modules.  Here we recall this notion and generalize it mildly.

\begin{definition}\label{def:tropker}
  Let $M$ be an $S$-module and $\varphi:V\to M$ a linear map.  The \emph{tropical kernel of $\varphi$} is
  the submodule $\tropker(\varphi)\subset V$ consisting of all those elements
  $\sum_{i=1}^n v_i e_i \in V$ such that
  \[
  \varphi\left(\sum_i v_i e_i\right) = \varphi\left(\sum_{i\neq j} v_i e_i \right)
  \]
for each $j \in [n]$.  Thus \[\tropker(\varphi) = (V^\vee/\bend(\varphi))^\vee.\]
\end{definition}

That is, $\tropker(\varphi) \subset V$ is the set of points $v\in V$ that satisfy all the relations in the congruence $\bend(\varphi)$, which is to say the set of $v\in V$ such that $\langle v, -\rangle: V^\vee \to S$ descends to $V^\vee/\bend(\varphi)$.
When $M$ is free, say of rank $m$, then the map $\varphi$ is given by an $m$-tuple of linear forms $\varphi_i\in V^\vee$ and $\tropker(\varphi)$ is the intersection of the corresponding tropical hyperplanes: 
\[\tropker(\varphi) = \bigcap_{i=1}^m \varphi_i^\perp.\]

In particular, if $f:M \to S$ is a linear form then the tropical kernel of $f$ is exactly the tropical hyperplane defined by $f$.

\begin{example}\label{example:trop-ker}
Let $S=\mathbb{T}$ and consider the linear map $\varphi: S^3 \to S^2$ defined by the matrix 
\[
\begin{pmatrix} 
0 & 1 & 2\\
0 & 0 & -\infty
\end{pmatrix},
\]
(note that 0 here is the multiplicative unit in $\T$).  The congruence $\bend(\varphi)$ is generated by the relations 
\begin{align*}
x_1 + 1x_2 &  \sim x_1 + 2x_3 \sim 1x_2 + 2x_3 \\
x_1  & \sim x_2.
\end{align*}
Looking at the $x_1=0$ slice, the first line of relations carves out a tropical line (Y-graph) with vertex at $(0,-1,-2)$, and the second cuts out the vertical line $x_2 = 0$, as shown in Figure \ref{fig:trop-ker-example}.  The tropical kernel of $\varphi$ is then the intersection of these two tropical hyperplanes, which is spanned by $(0,0,-1)$.  In this case the tropical kernel is indeed a tropical linear space, but if the second row of the matrix had instead been $(0, 1, -\infty)$ then the resulting tropical kernel would instead be the span of the vectors $(0,-1,-2)$ and $(0,-1,-\infty)$, corresponding to the lower leg of the Y, which is not a tropical linear space.
\end{example}

\begin{figure}\label{fig:trop-ker-example}
\begin{center}
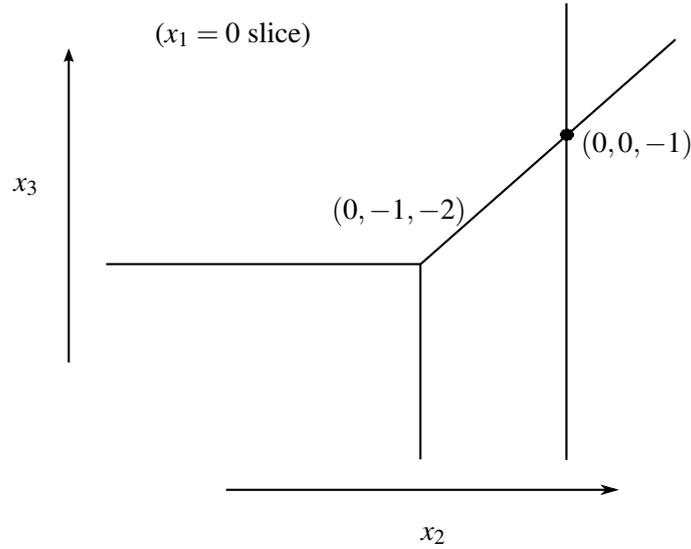
\end{center}
\caption{Illustration of the $e_1=0$ slice of two tropical hyperplanes whose intersection is the tropical kernel of the linear map from Example \ref{example:trop-ker}.}
\end{figure}


\section{An idempotent exterior algebra}

In this section we define the exterior algebra of a free module, and of a quotient of a free module, over an idempotent semifield, and we observe some (limited) similarities with the usual exterior algebra of modules over a ring.

\subsection{Exterior algebra of free modules}

Since bases for finitely generated free $S$-modules are unique up to rescaling and permutation (Proposition \ref{prop:GL}), we have:

\begin{lemma}\label{lem:basis}
The $S$-algebra congruence on $\Sym~V$ generated by $e^2_i \sim 0$ is independent of the choice of basis $\{e_i\}$.
\end{lemma}

This allows us to make the following definition:

\begin{definition}\label{def:TGA}
The \emph{tropical Grassmann algebra} of $V$ is the quotient of $\Sym~V$ by the $S$-algebra congruence generated by the relations in Lemma \ref{lem:basis}: \[ \ext V := \Sym~V/\langle e^2_i \sim 0 \rangle_{i=1}^n.\]
The grading on $\Sym~V$ descends to a grading on $\ext V$, and we call the direct summand $\ext^d V$ the $d^{\text{th}}$ \emph{tropical wedge power} of $V$.
\end{definition}

\begin{remark}
A few words are in order.
\begin{enumerate}
\item There is nothing inherently ``tropical'' in this construction, but we choose the name due to the relevance it will play for tropical linear spaces, discussed subsequently.
\item We use the notation $\ext V$ for the tropical Grassmann algebra and denote the product in it by $\wedge$ by analogy with the usual exterior algebra and wedge product over a ring, even though the multiplication here is commutative.  This commutativity is an artifact of the nonexistence of the element $-1$, the additive inverse of the multiplicative unit, in an idempotent semiring.
\item It is crucial that we only kill the squares of basis elements in this definition, since if $a+b = 0$ in an idempotent module then $a=b=0$.
\item This tropical $\ext$ only forms an endofunctor on the category of finite rank free $S$-modules with respect to the restricted class of module homomorphisms which send basis elements to scalar multiples of basis elements; these are given by matrices where each column has at most one nonzero entry.
\end{enumerate}
\end{remark}

The tropical Grassmann algebra $\ext V = \bigoplus_{d \ge 0}\ext^d V$ is free as an $S$-module, with the wedge power $\ext^d V$ free of rank $\binom{n}{d}$ for $d \le n$ and trivial for $d > n$; a basis is given by \[\{e_{i_1}\wedge \cdots \wedge e_{i_d}~|~1 \le i_1 \le \cdots \le i_d \le n\}\] as usual.  Since this wedge product is commutative, we shall abbreviate our notation by writing \[e_I := e_{i_1} \wedge \cdots \wedge e_{i_d}\] for any $I = \{i_1,\ldots,i_d\} \subset [n]$.  It is straightforward to see that the perfect pairing $V\times V^\vee \rightarrow S$ induces a perfect pairing $\ext^d V \times \ext^d (V^\vee) \rightarrow S$, so that $(\ext^d V)^\vee$ is canonically isomorphic to $\ext^d(V^\vee)$ and the elements $\{x_I\}$, where $x_I := x_{i_1} \wedge \cdots \wedge x_{i_d}$, form a dual basis to $\{e_I\}$.

Recall that the permanent of a square matrix is defined by the same formula as the determinant except with the sign of the permutation omitted \cite{Permanents}.  The permanent plays the role in the idempotent setting that the determinant does for rings and some authors prefer the term \emph{tropical determinant}; since there is no possibility for confusion, we will therefore use the term \emph{minor} to refer to the permanent of a square submatrix of a matrix over an idempotent semiring.

\begin{proposition}\label{prop:minors}
If  $v_j = \sum_{i=1}^n a_{ij}e_i$, for $1 \le j \le d$, then the coefficient of $e_I$ in $v_1\wedge \cdots \wedge v_d \in \ext^d V$ is the $I$-minor of the matrix $(a_{ij})$.  
\end{proposition}

\begin{proof}
This follows from the exact same formal manipulation as for the usual exterior algebra.
\end{proof}

\begin{remark}
Over a field, it is primarily this relation to determinants that renders the exterior algebra so useful, especially for studying linear dependence of vectors \cite{Grassmann}.  In the idempotent setting, the relevant replacement for linear dependence is encapsulated by the notion of a matroid and the geometry is that of tropical linear spaces.  As we shall see below, permanents only capture certain aspects of tropical linear dependence, but the tropical Grassmann algebra nonetheless plays a very similar role for tropical linear spaces as the usual exterior algebra does for linear spaces.
\end{remark}

\subsection{Exterior algebra of quotient modules}\label{sec:grassquot}
We now turn to the case of a module that is not necessarily free.  The role played by a special class of vectors, the basis vectors in the case of a free module, appears indispensable in our construction, so we shall restrict attention to modules over $S$ that are presented as a quotient of a free module.  This is not a particularly artificial restriction from the perspective of tropical geometry, where for instance the tropicalization of an affine variety depends on its embedding in affine space (cf., \cite{Payne}) and hence one studies coordinate algebras that are presented as quotients of free commutative algebras.

\begin{definition}\label{def:quotfree}
The \emph{tropical Grassmann algebra} of an $S$-module $M$, relative to a quotient presentation $V \twoheadrightarrow M$, is the tensor product \[\ext M := \ext V\otimes_{\Sym~V}\Sym~M.\]
\end{definition}

Concretely, this means we quotient the symmetric algebra $\Sym~M$ by setting the squares of the images of the basis vectors $e_i$ under the map $V \twoheadrightarrow M$ to be zero.  As before, this is graded, $\ext M = \bigoplus_{d \ge 0}\ext^dM$, and as in the case of commutative rings, this tensor product of commutative algebras over a semiring can be computed degree-wise:  $\ext^dM = \ext^dV\otimes_{\Sym^dV}\Sym^dM$.

We can also view $\ext M$ as a quotient of $\ext V$.  Later on, we shall need the following result which provides an explicit description of this latter presentation:

\begin{lemma}\label{lem:quot}
Suppose the congruence kernel of $V\twoheadrightarrow M$ is generated by the relations $\{u_i \sim v_i\}$. Then the congruence kernel of $\ext^d V \twoheadrightarrow \ext^d M$ is generated by the relations
\[
u_i \wedge e_I \sim v_i \wedge e_I
\]
for $I\in \binom{[n]}{d-1}$.
\end{lemma}
\begin{proof}
 Since $\Sym~M$ is presented as the free algebra $\Sym~V$ modulo the $S$-algebra congruence generated by the relations
  $u_i = v_i$, it follows that the $S$-module $\Sym^d M$ is presented as the free module $\Sym^d V$ modulo the $S$-module congruence generated by the relations produced by multiplying $u_i \sim v_i$ by a degree $d-1$ monomial in the $e_i$.  It then follows directly from properties of pushouts of modules that $\ext^d M$ is presented as $\ext^d V$ modulo the $S$-module congruence generated by the relations produced by multiplying $u_i \sim v_i$ by a square-free monomial in the $e_i$.
\end{proof}

Given a module $M$, the algebra $\ext M$ depends crucially on the choice of a presentation of $M$ as a quotient of a free module, even if $M$ itself is free, as this example shows.

\begin{example}
Consider $M=S^2$.  Relative to the presentation of $M$ as a trivial quotient of the free module $S^2$, we have $\ext^2 M \cong S$.  Now consider $M$ as a quotient of $S^3$ via the surjection $\pi: S^3 \to S^2=M$ corresponding to the matrix
\[
\begin{pmatrix}
1 & 1 & 0 \\
1 & 0 & 1
\end{pmatrix}
\]
with respect to bases $\{e_1, e_2, e_3\}$ and $\{f_1, f_2\}$.  
Relative to this presentation, the module $\ext^2 M$ is spanned by degree 2 monomials in the $f_i$ modulo the relations $\pi(e_i)^2 \sim 0$. Since $\pi(e_1) = f_1 + f_2$, we have $0\sim \pi(e_1)^2 = f_1^2 + f_2^2 + f_1 f_2$, and hence $\ext^2 M = 0$ since $a+b = 0$ implies $a = b = 0$ in an idempotent monoid.
\end{example}

However, a straightforward consequence of the preceding lemma is that we do have compatibility between the absolute and relative constructions if $M$ is free and the quotient is coordinate projection.

\begin{corollary}
If $V \twoheadrightarrow M$ is projection onto the free module spanned by a subset of the basis, then the tropical Grassmann algebra of $M$ coincides with that of $M$ relative to this presentation.
\end{corollary}


\section{Tropical Pl\"ucker vectors and tropical linear spaces}
We will now describe how tropical Pl\"ucker vectors and tropical linear spaces fit into
  the framework of the tropical Grassmann algebra in a way that is almost entirely parallel to the
  classical picture over a field.

\subsection{Tropical Pl\"ucker vectors} 

As defined by Speyer \cite{Speyer}, a rank $d$ tropical Pl\"ucker vector is an element
$(v_I) \in \mathbb{R}^{\binom{n}{d}}$ satisfying
the ``three-term tropical Pl\"ucker relations'': for every $J\in\binom{[n]}{d-2}$ and $\{i,j,k.l\}\subset J^c$, the maximum 
\[
\max \{v_{Jij}+v_{Jkl}, v_{Jik}+v_{Jjl}, v_{Jil}+v_{Jjk}\}
\]
is attained at least twice.  This was extended to the tropical numbers $\T = \mathbb{R}\cup\{-\infty\}$ where some coordinates are allowed to be infinite (see, e.g., \cite{Maclagan-Sturmfels}); in this generality, one must use all the quadratic relations obtained by tropicalizing the standard generating set for the classical ideal of Pl\"ucker relations.  By using the formalism of \S\ref{sec:pre-bend} we can extend this notion further to arbitrary $S$ and recast it in the setting of the tropical Grassmann algebra.

\begin{definition}
A rank $d$ \emph{tropical Pl\"ucker vector} is a nonzero $v = \sum v_Ie_I \in \ext^d V$ whose image under the natural map
$\ext^dV \rightarrow \Sym^2\ext^d V$
lies on the tropical hyperplanes defined by the functions
\[
\sum_{i \in A\smallsetminus B} x_{A-i}x_{B+i} \in \Sym^2 \left( \ext^d V^\vee \right), \quad A\in \binom{[n]}{d+1}, B\in \binom{[n]}{d-1}.
\]
Concretely, this means that the \emph{tropical Pl\"ucker relations} hold for $v$: 
\[
\sum_{i \in A\smallsetminus B} v_{A-i}v_{B+i} = \sum_{i \in A\smallsetminus B, i\neq p} v_{A-i}v_{B+i},
\]
where $A$ and $B$ range over the subsets indicated above, and $p$ ranges over the monomials terms on the left-hand side of this equality.
\end{definition}

Speyer noted that tropical Pl\"ucker vectors in his sense are the same as valuated matroids, supported on the uniform matroid, with coefficients in $\T$ \cite{Speyer,Dress-Wenzel}; when $S$ is totally ordered, the tropical Pl\"ucker vectors defined above are the same as valuated matroids, again in the sense of Dress and Wenzel, with coefficients in the abelian group $S^\times$.  In particular, when $S=\mathbb{B}$ is the booleans, rank $d$ tropical Pl\"ucker vectors are equivalent to rank $d$ matroids on the ground set $[n]$.  More precisely, for a rank $d$ matroid on $[n]$ defined by its set of bases $\mathcal{B} \subset \binom{[n]}{d}$, the corresponding tropical Pl\"ucker vector is $\sum_{I\in\mathcal{B}}e_I$, and an arbitrary element $\sum_{I\in \mathcal{I}}e_I\in\ext^d\mathbb{B}^n$ satisfies the tropical Pl\"ucker relations if and only if the collection $\mathcal{I} \subset \binom{[n]}{d}$ is the set of bases of a rank $d$ matroid on $[n]$.

\subsection{Tropical linear spaces}

The significance to Speyer of tropical Pl\"ucker vectors is that they determine $\T$-modules called
``tropical linear spaces,'' which are a common generalization of the Bergman fan of a matroid \cite{Bergman, Feichtner-Sturmfels} and of the tropicalization of a linear subspace over a valued field \cite{Speyer,Speyer-Sturmfels}.  In this section we set out a straightforward reformulation of Speyer's construction in terms of the tropical Grassmann algebra.

Classically, which is to say over a field $k$, one recovers a $d$-dimensional linear subspace $W \subset k^n$ as the kernel of the wedge-multiplication map 
\[
- \wedge w : k^n \rightarrow \ext^{d+1} k^n,
\]
where $w \in \ext^d k^n \cong k^{\binom{n}{d}}$ is the Pl\"ucker vector of $W$, which is well-defined up to multiplication by a nonzero scalar.  This won't work in the idempotent setting if applied verbatim; for instance, we recover a 1-dimensional subspace as the wedge-kernel if and only if it is spanned by a basis vector (for example, if $v = e_1+e_2$ then $v\wedge v = e_1\wedge e_2 \ne 0$, whereas if $v = e_1$ then $v\wedge v = 0$).  The key is to replace the kernel with the tropical kernel introduced in \S\ref{sec:pre-ker}.

The tropical linear space $L_w\subset \T^n$ associated
to a tropical Pl\"ucker vector $w=(w_I)\in \T^{\binom{n}{d}}$ is by definition \cite{Speyer,Maclagan-Sturmfels} the intersection of the tropical
hyperplanes defined by the linear forms
\[
\sum_{i\in J} w_{J-i} x_i\in (\T^n)^\vee
\] 
for all $J\in \binom{[n]}{d+1}$.  Each of these linear forms is a valuated circuit vector of $w$, and each circuit appears at least once; associated with a $(d+1)$-element set $J$ is the unique circuit contained in $J$.  We can extend this definition without any trouble to define tropical linear spaces $L_w \subset V$ for a free module over an arbitrary idempotent semifield $S$.  With the tropical kernel and tropical Grassmann algebra replacing their classical counterparts, we see that this definition of Speyer is indeed what one would hope for based on the classical situation.

\begin{proposition}\label{prop:wedgekerTLS}
The tropical linear space $L_w \subset V$ associated to a tropical Pl\"ucker vector 
\[
w = \sum_{I \in \binom{[n]}{d}} w_Ie_I \in \ext^d V
\]
is the tropical kernel of the wedge-multiplication map
\[
- \wedge w :V \rightarrow \ext^{d+1} V.
\] 
\end{proposition}

\begin{proof}
We must show that, for each $J\in \binom{[n]}{d+1}$, the $J$-component of the homomorphism $-\wedge w$ is the linear form $\sum_{i\in J} w_{J-i} x_i$.  Given any $v =\sum_{i=1}^n v_i e_i \in V$, we have 
\[
v\wedge w = \sum_{i=1}^n \sum_{I\in \binom{[n]}{d}} v_i w_I e_I \wedge e_i = \sum_{J \in
  \binom{[n]}{d+1}} \sum_{i\in J} v_i w_{J-i} e_J,
\]
since $e_I\wedge e_i =0$ when $i\in I$, as desired.
\end{proof}

Note that $L_w = L_{sw}$ for all nonzero $s\in S$.  For this reason we will often view tropical Pl\"ucker vectors as elements of the projectivization $\PP\left(\ext^d V\right)$, just as in the classical setting.

\subsection{Submodules versus quotients}\label{sec:Qw}

Having just seen that the passage from a tropical Pl\"ucker vector to the corresponding tropical linear space proceeds quite analogously to the classical situation, it is natural to ask about the converse.  We will address this fully in \S\ref{sec:topwedge}, but for now we note an important contrast to the classical situation that requires a slight change in perspective.

Over a field $k$, given a $d$-dimensional subspace $L\subset k^n$, the space $\ext^d L$ is a
line in $\ext^d k^n$, and the Pl\"ucker vector of $L$ spans this line.  We will be concerned
with a tropical analogue of this picture, but to proceed we must deal with the fact that the
tropical Grassmann algebra is defined \emph{only for quotients of free modules} (\S\ref{sec:grassquot}), and in particular, it is not defined for submodules of free modules.  

\begin{definition}
The \emph{tropical quotient module} associated to an element $w\in \ext^d V$ is
\[
Q_w := V^\vee /\bend(-\wedge w),
\]
where $\bend(-\wedge w)$ denotes the congruence defined in \S\ref{sec:pre-ker} generated by the bend relations of the linear map $-\wedge w : V \rightarrow \ext^{d+1}V$.
\end{definition}

Let us set $L_w := \tropker(-\wedge w)$.  By Proposition \ref{prop:wedgekerTLS}, if $w$ is a tropical Pl\"ucker vector then $L_w$ is the associated tropical linear space, but for any $w$ we can consider the submodule $L_w \subset V$ and the quotient module $V^\vee \twoheadrightarrow Q_w$.  While the former is very natural to study from the geometric perspective, one of the insights of this paper is that, from the algebraic perspective the latter is a more natural object to work with.

\begin{example}\label{ex:graphic}
Consider the graphic matroid $M(K_4)$ on the complete graph with four vertices.  This has a geometric representation given by the following diagram \cite[Appendix p.640]{Oxley}:
\begin{center}
\begin{figure}[h!]
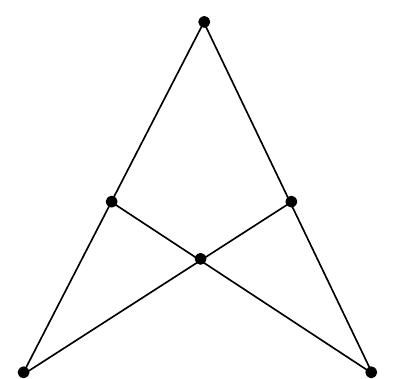
\caption{A geometric representation of the graphic matroid $M(K_4)$.}
\end{figure}
\end{center}
The bases are therefore the 16 triples\[\mathcal{B} = \binom{[6]}{3} \setminus \{\{1,2,3\},\{1,4,5\},\{2,5,6\},\{3,4,6\}\}.\]  The tropical Pl\"ucker vector is $w = \sum_{I\in\mathcal{B}}e_I \in \ext^3\mathbb{B}^6$, and the linear map \[-\wedge w: \mathbb{B}^6 \rightarrow \ext^4\mathbb{B}^6 \cong \mathbb{B}^{15}\] is given by the matrix 
\[
\small{
\left(\begin{array}{c|cccccc}
 & e_1 & e_2 & e_3 & e_4 & e_5 & e_6\\
\hline
e_{1234} & 0 & 0 & 0 & -\infty & -\infty & -\infty\\
e_{1235} & 0 & 0 & 0 & -\infty & -\infty & -\infty\\
e_{1236} & 0 & 0 & 0 & -\infty & -\infty & -\infty\\
e_{1245} & 0 & -\infty & -\infty & 0 & 0 & -\infty\\
e_{1246} & 0 & 0 & -\infty & 0 & -\infty & 0\\
e_{1256} & -\infty & 0 & -\infty & -\infty & 0 & 0\\
e_{1345} & 0 & -\infty & -\infty & 0 & 0 & -\infty\\
e_{1346} & -\infty & -\infty & 0 & 0 & -\infty & 0\\
e_{1356} & 0 & -\infty & 0 & -\infty & 0 & 0\\
e_{1456} & 0 & -\infty & -\infty & 0 & 0 & -\infty\\
e_{2345} & -\infty & 0 & 0 & 0 & 0 & -\infty\\
e_{2346} & -\infty & -\infty & 0 & 0 & -\infty & 0\\
e_{2356} & -\infty & 0 & -\infty & -\infty & 0 & 0\\
e_{2456} & -\infty & 0 & -\infty & -\infty & 0 & 0\\
e_{3456} & -\infty & -\infty & 0 & 0 & -\infty & 0
\end{array}\right).
}
\]
Each row of this matrix determines a linear form corresponding to a circuit of the matroid, with each circuit occuring at least once.  The rows corresponding to a given circuit $C$  are in bijection with the sets that are one larger than a basis and contain $C$.
For instance, each of the  first three rows yields the form $x_1+x_2+x_3$, which corresponds to the circuit $\{1,2,3\}$, and this circuit is contained in three subsets of size 4: $\{1,2,3,4\}, \{1,2,3,5\},$ and  $\{1,2,3,6\}$. The tropical linear space $L_w \subset \mathbb{B}^6$ is the intersection of the 15 tropical hyperplanes defined by these linear forms (of course, there are redundancies in this intersection).  On the other hand, the module $Q_w$ is the quotient of the free module on $x_1,\ldots,x_6$ by the bend relations determined by each of these linear forms; for instance, the first row yields the relations $x_1+x_2 \sim x_1+x_3 \sim x_2+x_3$.  We shall return to this example later.
\end{example}

\begin{remark}\label{rem:bend-duality-conjecture}
As discussed in \S\ref{sec:pre-ker}, we have $Q_w^\vee \cong L_w$ for any $w$.  When $d=n-1$, so that $L_w$ is a tropical hyperplane, we proved in \cite[Theorem 5.2.2]{GG1} that $Q_w \cong
L_w^\vee$, and it seems likely that this holds for all $d$, at least when $S$ is in a certain class
of well-behaved idempotent semifields that includes $\T$ and $\mathbb{B}$.  Frenk showed that if $w$ is the tropical Pl\"ucker vector of any non-valuated matroid, then the $\mathbb{B}$-module $L_w^\vee$ is in bijection with the flats of this matroid \cite[Remark on p.89]{Frenk}.
\end{remark}

In the next subsection we shall study the notion of tropical orthogonality and use it to explore another relation between $L_w$ and $Q_w$.

\subsection{Dualities, the tropical Hodge star, and generators for tropical linear spaces}\label{sec:duals}

Here we review the duality theory of tropical linear spaces and discuss its relation with a tropical Hodge star operator that we introduce.

Given a submodule $L\subset V$, it's \emph{orthogonal dual} $L^\perp \subset V^\vee$ is the submodule defined by
\[
L^\perp = \bigcap_{\alpha \in L} \tropker(\alpha).
\]
Note that for arbitrary submodules, we always have $L\subset L^{\perp \perp}$, although this inclusion need not be an equality unless $L$ is a tropical linear space, as we will see below.

\begin{example}
Consider the submodule $L=\mathrm{span}\: \{e_1+e_2, e_2+e_3\}$ in $V=S^3$.
We then have $L^\perp=\mathrm{span}\:\{x_1+x_2+x_3\}$, but $L^{\perp\perp}=\mathrm{span}\:\{e_1+e_2, e_2+e_3, e_1 + e_3\}$, so $L$ is strictly contained in $L^{\perp \perp}$.
\end{example}

Fix an identification $\ext^n V \cong S$.  The wedge product of complementary degrees gives a perfect pairing $\ext^d V \times \ext^{n-d} V \to \ext^n V \cong S$, and so we obtain an identification
\begin{equation}
\star : \ext^{n-d}V \cong \ext^d V^\vee.
\end{equation}
which we will call the \emph{tropical Hodge star}.  Concretely, fixing $e_1 \wedge \cdots \wedge e_n$ as the generator of $\ext^n V$, the map $\star$ is given by $e_I \mapsto x_{I^c}$, where $I^c := [n]\smallsetminus I$.   

Dress \cite{Dress} and Speyer \cite{Speyer} introduced a duality for valuated matroids, extending the orthogonal dual for ordinary matroids; in terms of the Hodge star, it defined as
\[
w\mapsto \star w.
\]
One can verify that if $w$ satisfies the tropical Pl\"ucker relations then so does $\star w$.

Consider the maps 
\[
F = -\wedge w: V \to \ext^{d+1} V, 
\]
and 
\[
G = -\wedge \star w: V^\vee \to \ext^{n-d+1} V^\vee.
\]
The image of the dual $F^\vee$ is spanned by the vectors 
\[
\alpha_J = \sum_{i\in J} w_{J-i}x_i \in V^\vee,
\]
for $J\in \binom{[n]}{d+1}$, and the image of $G^\vee$ is spanned by the vectors 
\[
\beta_K = \sum_{i \in K^c} w_{K+i}e_i \in V,
\]
for $K\in \binom{[n]}{d-1}$.  We will occassionally write $\alpha(w)_J$ and $\beta(w)_K$ when
we need to make the dependence on $w$ explicit. Note that the $\alpha$ vectors associated with $w$ are precisely equal to the $\beta$ vectors associated with $\star w$, and vice versa, via the correspondences
\[
\alpha(w)_J = \beta(\star w)_{J^c} \text{ and } \beta(w)_K = \alpha(\star w)_{K^c}.
\]
When $w$ is a tropical Pl\"ucker vector the vectors $\alpha_J$ are the called the \emph{valuated circuits} of $w$ and the vectors $\beta_K$ are called the \emph{valuated cocircuits} of $w$.

\begin{proposition}\label{prop:inclusions}
The following are equivalent:
\begin{enumerate}
\item $w$ is a tropical Pl\"ucker vector,
\item $\im(F^\vee) \subset \tropker (G)$,
\item $\im(G^\vee) \subset \tropker (F)$.
\end{enumerate}
\end{proposition}
\begin{proof}
Since $\tropker(G) = \bigcap_K \beta_K^\perp$ and $\im(F^\vee)$ is spanned by the vectors $\alpha_J$, we have $\im(F^\vee)\subset \tropker(G)$ if and only if $\alpha_J \in \beta_K^\perp$ for all $J$ and $K$, which is equivalent to $\beta_K \in \alpha_J^\perp$ for all $J,K$.   These conditions are precisely the tropical Pl\"ucker relations.
\end{proof}

By definition, regardless of whether or not $w$ is a tropical Pl\"ucker vector, the vectors $\alpha_J$ give the equations that describe $L_w = \tropker (F)$, which is to say that $\tropker(F) = \cap_J \alpha_J^\perp$.  A fundamental result about tropical linear spaces (first proved in \cite[Theorem 3.8]{Murota-circuits}, and later given a different proof in \cite[Proposition 4.1.9]{Frenk}) is that when $w$ is a tropical Pl\"ucker vector then the inclusions of Proposition \ref{prop:inclusions} above are actually equalities; i.e.,

\begin{theorem}\label{thm:spanned-by-cocircuits}
Assume $S$ is totally ordered. If $w$ is a tropical Pl\"ucker vector then the valuated cocircuit vectors $\beta_K$ span the associated tropical linear space $L_w$, and so $\im(F^\vee) = \tropker (G)$ and $\im(G^\vee) = \tropker (F)$.
\end{theorem}

A standard corollary of this result is the following statement.

\begin{corollary}
Let $S$ be a totally ordered semifield. Let $w \in \ext^d V$ be a tropical Pl\"ucker vector.  Then \[(L_w)^\perp = L_{\star w},\] and hence $L_w^{\perp\perp} = L_w$.
\end{corollary}
\begin{proof}
If $L_w$ is spanned by the valuated cocircuit vectors $\beta_K$ then $L_w^\perp = \cap_K \beta_K^\perp = L_{\star w}$.  Thus $L_w^{\perp \perp} = L_{\star w}^\perp = L_{\star \star w} = L_w$.
\end{proof}

From the above facts we obtain direct formulae for passing between a tropical linear space $L_w$ and the associated quotient module $Q_w$.

\begin{corollary}
Let $S$ be a totally ordered semifield. When $w\in\ext^d V$ is a tropical Pl\"ucker vector then $L_w = Q_w^\vee$, and the module $Q_w$ can be recovered directly
from the associated tropical linear space $L_w$ by the formula 
\[
Q_w = V^\vee / \bend(L_w^\perp).
\]
\end{corollary}
\begin{proof}
By definition, $Q_w$ is the quotient of $V^\vee$ by the bend relations of the image of $F^\vee$, where $F=-\wedge w$, and
by Theorem \ref{thm:spanned-by-cocircuits}, the image of $F^\vee$ spans $L_w^\perp$.
\end{proof}
As mentioned in Remark \ref{rem:bend-duality-conjecture}, when $d=n-1$, and conjecturally for all $d$, we also have the formula $Q_w = L_w^\vee$.

\section{Stable sums of tropical linear spaces}\label{sec:stablesum}

Speyer defined the stable intersection of tropical linear spaces in essence by tropicalizing the formula for the Pl\"ucker vector of a transverse intersection of linear subspaces.  He showed that this yields an intersection of tropical linear spaces which always yields the expected dimension, when it is non-negative, at least in the tropical torus $\mathbb{R}^n \subset \T^n$ \cite[\S3]{Speyer}.  A stable intersection with this dimension property extends to tropical varieties of arbitrary degree \cite{Sturmfels-first-steps,Mikhalkin-ICM}, and by using orthogonal duality of tropical linear spaces it determines a dimension-additive stable sum of tropical linear spaces, which has appeared in \cite{Fink-Rincon}.  In this section we first translate the stable sum into our tropical Grassmann setting and then explore some consequences and related constructions.

\subsection{Stable sum} 

Let $w\in \ext^d V$ and $w'\in \ext^{d'} V$ be tropical Pl\"ucker vectors such that there exists nonzero coordinates $w_I, w'_J$ with disjoint indices $I,J\subset [n]$ (so in particular, $d+d' \leq n$).  The \emph{stable sum} $L_w\oplus_{st}L_{w'}$ of the corresponding tropical linear spaces has tropical Pl\"ucker coordinates with $K^{\text{th}}$ entry 
\begin{equation}\label{eq:stablesum}
  \sum_{I\sqcup J = K} w_I w'_J,
\end{equation}
where the sum is over all partitions of $K$ into disjoint subsets $I$ and $J$ of size $d$ and $d'$. This formula has a simple description in terms of the tropical Grassmann algebra.

\begin{proposition}\label{prop:sumwedge}
For tropical Pl\"ucker vectors $w\in \ext^d V$ and $w'\in \ext^{d'} V$ as above, the stable sum $L_w\oplus_{st}L_{w'}$ has tropical Pl\"ucker vector $w\wedge w' \in \ext^{d+d'} V$.
\end{proposition}

\begin{proof}
Write $w = \sum w_I e_I$ and $w' = \sum w'_J e_J$.  Then 
\[
  w\wedge w' = \sum_{I,J} w_I w'_J e_I \wedge e_J.
\]
The only nonzero terms here are those for which $I$ and $J$ are disjoint.
\end{proof}
(Note that $w\wedge w' = 0$ if and only if the disjoint index condition on the pair $w,w'$ is not satisfied.)

Formula \eqref{eq:stablesum} first appeared in \cite[\S3.1]{Fink-Rincon}, and the fact that it indeed produces a tropical Pl\"ucker vector follows immediately from \cite[Proposition 3.1]{Speyer} in the case when the ground semifield $S$ is $\mathbb{T}$, or more generally, as long as $S$ is totally ordered.  Speyer's arugment is written in terms of matroid polytopes, so it does not directly generalize to arbitrary semifields.  Here we provide a different proof that works over any semifield $S$.

\begin{proposition}
If $w\in \ext^d V$ and $w'\in \ext^{d'}V$ are tropical Pl\"ucker vectors with nonzero wedge product, then $w\wedge w' \in
\ext^{d+d'} V$ is also a tropical Pl\"ucker vector.  
\end{proposition}

\begin{proof}
By definition,  $w\wedge w'$ is a tropical Pl\"ucker vector if for any $A\subset \binom{[n]}{d+d'+1}$ and $B\subset \binom{[n]}{d+d'-1}$, the bend relations of the expression
\[
\sum_{i\in A\smallsetminus B} (w\wedge w')_{A-i} (w\wedge w')_{B+i}
\]
hold, which is to say that any single term in the summation can be omitted without changing the value of the expression (as an element of $S$).  The above formula can be expanded out in terms of the components of $w$ and $w'$ as
\begin{equation}\label{eq:wedge-expansion1}
\sum_{i\in A\smallsetminus B} \:\: \sum_{\substack{C \sqcup C' = A-i \\ D \sqcup D' = B+i}} (w_C w_D) (w'_{C'} w'_{D'}).
\end{equation}
This can be rewritten as
\begin{align*}
\sum_{\substack{E\sqcup C' = A \\ F\sqcup D' = B}} (w'_{C'} w'_{D'})\left(\sum_{i \in E\smallsetminus F} w_{E-i} w_{F+i} \right)\\
+ \sum_{\substack{C \sqcup G = A \\ D \sqcup H = B}} (w_C w_D) \left(\sum_{i \in G\smallsetminus H} w'_{G-i} w'_{H+i} \right),
\end{align*}
where the terms on the first line correspond to the terms in \eqref{eq:wedge-expansion1} where
$i\in D$ (by setting $E = C+i$ and $F=D-i$), and the second line corresponds to the terms where $i
\in D'$ (by setting $G=C'+i$ and $H=D'-i$).  Since $w$ and $w'$ are tropical Pl\"ucker vectors, for
any $i_0 \in A\smallsetminus B$, omitting all of the $i=i_0$ terms leaves the value of this expression
unchanged.  But the sum of these omitted terms is precisely the $i=i_0$ term in the expression
\eqref{eq:wedge-expansion1}, and so it follows that $w\wedge w'$ is indeed a tropical Pl\"ucker vector.
\end{proof}

\subsection{Stiefel spaces and transversal matroids}

Fink and Rinc\'on define a \emph{Stiefel} tropical linear space to be any tropical linear space associated to the vector of maximal minors of a tropical matrix, assuming not all these minors vanish \cite{Fink-Rincon}.  They observe that such a vector indeed satisfies the tropical Pl\"ucker relations by choosing a lift of the tropical matrix to a valued field and then noting that if this lift is sufficiently generic then valuation commutes with taking minors.  This argument applies to any idempotent semifield $S$ for which there exists a field $k$ and surjective valuation $\nu : k \twoheadrightarrow S$ with sufficiently large fibers.  The following result characterizes the tropical Pl\"ucker vectors arising this way in terms of their tensor decomposition in the Grassmann algebra.  We call an element $w\in \ext^d V$ \emph{totally decomposable} if it factors into a wedge product of $d$ elements of $V$.

\begin{proposition}\label{prop:totaldecomp}
An element of $\ext^d V$ is totally decomposable if and only if its coordinates in the standard basis are the maximal minors of a $d\times n$ matrix.  Thus, the locus of Stiefel tropical linear spaces is in natural bijection with the locus of totally decomposable tensors in $\PP\left(\ext^dV\right)$.
\end{proposition}

\begin{proof}
The first statement follows immediately from Proposition \ref{prop:minors}, and the second statement follows immediately from the first.
\end{proof}

\begin{remark}
This result is in stark contrast to the classical case where every Pl\"ucker vector is totally decomposable.  Indeed, over a field $k$, the Pl\"ucker vector of a linear subspace is the wedge product of the vectors in any basis, and the bijection between linear subspaces and totally decomposable tensors yields the Pl\"ucker embedding $\Gr(d,n) \hookrightarrow \PP\left(\ext^d k^n\right)$ of the Grassmannian.  It would be interesting to study tensor factorization of tropical Pl\"ucker vectors with the goal of understanding when a tropical linear space decomposes as a non-trivial stable sum.
\end{remark}

Since the tropical Pl\"ucker relations are trivial in rank one, and any wedge product of tropical Pl\"ucker vectors is a tropical Pl\"ucker vector if it is nonzero, Proposition \ref{prop:totaldecomp} gives an alternate proof, extending to arbitrary $S$ and avoiding the use of valuations, of the statement that the maximal minors of a tropical matrix satisfy the tropical Pl\"ucker relations.  By combining with Proposition \ref{prop:sumwedge}, it also provides an alternate proof of \cite[Proposition 3.4]{Fink-Rincon}, which says that the condition of being Stiefel is equivalent to being a stable sum of 1-dimensional spaces.  Specializing to the case $S=\mathbb{B}$, where the stable sum becomes the usual matroid union, we see from this discussion and \cite[Proposition 12.3.7]{Oxley} that a matroid is transversal if and only if, when viewed as a tensor, it is totally decomposable.

\subsection{Elongation}
Murota in \cite[\S2]{Murota-valuated} extends the notion of matroid elongation to valuated matroids as follows.  Given integers $1 \le d \le d' \le n$ and a rank $d$ tropical Pl\"ucker vector $(w_I) \in S^{\binom{[n]}{d}}$, the \emph{elongation} is the rank $d'$ tropical Pl\"ucker vector defined by $w'_J = \sum_{I \subset J}w_I$.  That this formula yields a vector satisfying the tropical Pl\"ucker relations is \cite[Theorem 2.1(2)]{Murota-valuated}.
\begin{proposition}
The elongation of the tropical Pl\"ucker vector $w = \sum w_I e_I \in \ext^d V$ is \[w\wedge\left(\sum_{|K| = d'-d}e_K\right) \in \ext^{d'}V.\]
\end{proposition}

\begin{proof}
This follows immediately from the fact that $e_I\wedge e_K =0$ if and only if $I\cap K \ne \varnothing$.
\end{proof}

\begin{remark} 
This gives a less \textit{ad hoc} proof that the tropical Pl\"ucker relations are satisfied for elongation, since $\sum e_K$ is clearly a tropical Pl\"ucker vector, in addition to the geometric interpretation furnished by Proposition \ref{prop:sumwedge}: the elongation of a tropical linear space is its stable sum with the tropical linear space associated to the uniform matroid. 
\end{remark}


\section{Top exterior powers}\label{sec:topwedge}

In this section we will show how the tropical Grassmann algebra provides a novel algebraic reformulation of the tropical
Pl\"ucker relations.  This also provides a direct way of recovering the tropical Pl\"ucker vector
from its associated tropical linear space.

\subsection{A reformulation of the tropical Pl\"ucker relations}

Recall that in \S\ref{sec:Qw} we have defined, for any $w = \sum w_Ie_I \in \ext^d V$, a quotient module $Q_w := V^\vee /\bend(-\wedge w)$.

\begin{theorem}\label{thm:main}
A nonzero $w\in \ext^d V$ satisfies the
tropical Pl\"ucker relations if and only if the module $\ext^d Q_w$ is free of rank one.
\end{theorem}

Before presenting the proof, let us first note that the theorem is only about quotients arising from wedging with a $d$-multivector.  Indeed, there exist quotients $V^\vee \to (V^\vee / \sim) = Q$ such that $\ext^d Q$ is free of rank one but $Q$ is not of the form $Q_w$ for any $w\in \ext^d V$.
\begin{example}
Consider $V^\vee$ free rank 3 with basis $x_1, x_2, x_3$ and let $Q$ be the quotient of $V$ by the relations
\[
x_1 + x_2 = x_1 + x_3 \text{ and } x_2 + x_3 = x_2.
\]
Observe that $\ext^2 Q$ is free rank 1, spanned by $x_1 x_2 = x_1 x_3$ with $x_2 x_3 = 0$. This is because in $\ext^2 Q$ we have
\[
x_1 x_2 = x_1(x_1 + x_2) = x_1(x_1 + x_3) = x_1 x_3
\]
and
\[
x_2x_3 = x_2( x_2 + x_3) = x_2^2 = 0,
\]
and by enumerating all other relations one can check that they all can be deduced from these. On the other hand, if $w \in \ext^2 V$ is nonzero then $-\wedge w : V \rightarrow \ext^3 V \cong S$ is a linear form and $L_w = Q_w^\vee$ is a tropical hyperplane, whereas $Q^\vee$ is clearly not, so we cannot have $Q = Q_w$ for any 2-multivector $w$.
\end{example}

Moving on to the proof, we begin with a couple lemmas.

\begin{lemma}\label{lem:presentation}
  For any nonzero $w\in \ext^d V$, the module $\ext^d Q_w$ is presented as the quotient of
  $\ext^d V^\vee$ by the congruence generated by the bend relations of the expressions
  \begin{equation}\label{eq:general-xw-expression}
    \sum_{i\in A\smallsetminus B} w_{A-i} x_{B+i} 
  \end{equation}
  for $A \in \binom{[n]}{d+1}$ and $B\in \binom{[n]}{d-1}$.  In particular, taking $B=A\smallsetminus\{p,q\}$
  yields the relation
  \begin{equation}\label{eq:key-rel}
    w_{A-p}x_{A-q} = w_{A-q}x_{A-p}.
  \end{equation}
\end{lemma}
\begin{proof}
  We saw in the proof of Proposition \ref{prop:wedgekerTLS} that the module $\im((-\wedge w)^\vee) \subset V^\vee$ is spanned by the linear forms $\sum_{i\in A} w_{A-i} x_j$, so by definition the congruence $\bend(-\wedge w)$ on $V^\vee$ is generated by the bend relations of these linear forms.
  Thus, by Lemma \ref{lem:quot}, the congruence kernel of $\ext^d V^\vee \to \ext^d
  Q_w$ is generated by the bend relations of the expressions of the form
\[
x_B \wedge \left(\sum_{i\in A} w_{A-i} x_i\right)
\]
for $B\in \binom{[n]}{d-1}$.  Since the squares of the $x_i$ vanish, this expression becomes $
\sum_{i\in A\smallsetminus B} w_{A-i} x_{B+i}.$ 
\end{proof}

\begin{lemma}\label{lem:vanishing}
For any nonzero $w\in \ext^d V$, in $\ext^d Q_w$ we have that $w_I=0$ implies $x_I=0$, and the converse holds if at least one $x_I$ is nonzero.
\end{lemma}

\begin{proof}
  We first show that $w_I=0$ implies $x_I=0$.  Let $\supp(w) := \{I\in \binom{[n]}{d}~|~w_I \ne 0\}$.  Suppose $I\notin \supp(w)$, and define
\[
\|I\|_w := \min_{K\in \supp(w)} |I\smallsetminus K|.
\]
We will prove that $x_I=0$ for all $I\notin \supp(w)$ by induction on $\|I\|_w$.  If $\|I\|_w=1$ then there exists $K\in \supp(w)$ with $|I\smallsetminus K| = 1$.  Then \eqref{eq:key-rel} gives
\[
w_Kx_I = w_Ix_K.
\] 
Since $w_K \ne 0$ and $w_I = 0$, it follows that $x_I=0$.  Now suppose $x_I=0$ for all $I$
with $0 < \|I\|_w < m$.  Given $I$ such that $\|I\|_w=m$, let $K\in \supp(w)$ be an element such
that $|I\smallsetminus K| = m$, and consider the bend relations of the expression
\eqref{eq:general-xw-expression} for $A=K+j$ and $B=I-j$ for some $j\in I \smallsetminus K$.  This gives the bend
relations of the expression
\[
w_K x_I + \sum_{i \in K \smallsetminus I} w_{K+j-i} x_{I+i-j}. 
\]
Since $\|I+i-j\|_w=m-1$, the inductive hypothesis says that each term in this summation over $K\smallsetminus I$ is zero. For the bend relations of the full expression to be satisfied, we must then have that $w_K x_I=0$, and since $w_K \neq 0$, this implies that $x_I=0$.

The other direction follows from essentially the same argument with the roles of $x$ and $w$
reversed. Let $\supp(x) := \{I\in \binom{[n]}{d} ~| ~ x_I\neq 0\}$, which we assume is nonempty, and
define $\|I\|_x := \min_{K\in \supp(x)}|I\smallsetminus K|$.  If $\|I\|_x = 1$ then there is some
$K\in \supp(x)$ realizing this distance, and in the relation \[ w_Ix_K = w_K x_I \] $x_I$ is zero
but $x_K$ is nonzero, so $w_I$ must be zero. Now suppose inductively that $w_I = 0$ for all $I$ with
$0< \|I\|_x < m$. Given an $I$ such that $\|I\|_x = m$, let $K\in \supp(x)$ realize this distance
and set $A=I+j$ and $B=K-j$ for some arbitrary $j\in K\smallsetminus I$. Then the bend relations of
the expression 
\[ 
w_I x_K + \sum_{i \in I \smallsetminus K} w_{I+j-i} x_{K+i-j} 
\]
are satisfied. For each $i$, $\|I+j-i\|_x = m-1$, so $w_{I+j-i} = 0$ by inductive hypothesis, and $x_K$ is nonzero,
so it must be the case that $w_I=0$. 
\end{proof}

\begin{proof}[Proof of the main theorem]
  The argument is by manipulation of the presentation from Lemma \ref{lem:presentation}.  As in the preceding proof, we set $\supp(w) := \{I\in \binom{[n]}{d}~|~w_I \ne 0\}$.

  Suppose first that $w$ is a tropical Pl\"ucker vector, so by definition, for any $X,Y$ with $|Y| = d+1$ and
  $|X|=d-1$, the bend relations of the summation
\[
\sum_{i\in Y\smallsetminus X} w_{Y-i}w_{X+i}
\]
are satisfied.

We first show that these tropical Pl\"ucker
relations ensure that if $I$ and $I'$ are both in $\supp(w)$ then
\begin{equation}\label{eq:scalar-rel}
x_{I'} = (w_{I'}/w_I)x_I,
\end{equation}
so $x_I$ and $x_{I'}$ are identified up to a scalar; by Lemma \ref{lem:vanishing}, $x_J = 0$ for all $J\notin\supp(w)$, so we will conclude that $\ext^d Q_w$ is generated by a single element.

From the form of the tropical Pl\"ucker relations above one can show that if $I,I' \in \supp(w)$ are
such that $|I\smallsetminus I'| = m$ then there exists at least one pair $J,J'\in \supp(w)$ such
that
\begin{align*}
|I\smallsetminus J| = |I' \smallsetminus J'|=1, \\
|I \smallsetminus J'| = |I' \smallsetminus J| = m-1.
\end{align*}
To see this, take $X=I-j$ and $Y=I'+j$ for some $j\in I\smallsetminus I'$. The $i=j$ term in the
above sum is $w_Iw_{I'}$, and the remaining terms are all of the form $w_J w_{J'}$ for $J,J'$ as
above, and the bend relations imply that if one term in the sum is nonvanishing then at least one
other must also be nonvanishing. Repeating this $m$ times, we find a chain from $I$ to $I'$ along
which \eqref{eq:key-rel} can be iteratively applied to yield \eqref{eq:scalar-rel}.  Thus $x_{I}$
and $x_{I'}$ are proportional if $I,I'\in \supp(w)$.  Therefore $\ext^d Q_w$ is generated by
a single element.

It remains to show that it is in fact free. By Lemma \ref{lem:presentation}, it suffices to show that
when we use the equation \eqref{eq:key-rel} to transform each expression \eqref{eq:general-xw-expression} into an
expression purely in terms of some single $x_{I_0}$, the coefficient will be expressed as a sum for which the bend relations are satisfied.  Consider an expression
\[
\sum_{i\in A\smallsetminus B} w_{A-i} x_{B+i};
\]
If all of the $x_{B+i}$ appearing in the sum vanish then the entire expression is zero and hence its
bend relations cannot yield a nontrivial relation between distinct multiples of any $x_I$, so assume there is some $j\in A\smallsetminus B$ for which $x_{B+j} \ne 0$.  Multiplying by $w_{B+j}$ and using \eqref{eq:key-rel} yields the expression
\[
\left(\sum_{i\in A\smallsetminus B} w_{A-i} w_{B+i} \right) x_{B+j},
\]
and the bend relations of the sum appearing as the coefficient here are indeed satisfied precisely
because $w$ is a tropical Pl\"ucker vector.  Thus we have shown that $\ext^d L_w$ is indeed
free rank 1.

Now assume that $\ext^d Q_w$ is free of rank one.  There must exist at least one
$I_0\in \binom{[n]}{d}$ such that $x_{I_0}$ is nonzero in $\ext^d Q_w$, and any other monomial
$x_{I'}$ must be equal to a (possibly zero) scalar times $x_{I_0}$.  In the expression \eqref{eq:general-xw-expression}, if there is some $j$ such that $x_{B+j}$ is nonzero then multiplying through by $w_{B+j}$ and applying \eqref{eq:key-rel}, we find that the bend relations of the expression
$\sum_{i\in A \smallsetminus B} w_{A-i} w_{B+i}x_{B+j}$
are satisfied.  By the hypothesis that $\ext^d Q_w$ is free, it follows that the bend relations of 
\[
\sum_{i\in A \smallsetminus B} w_{A-i} w_{B+i}
\]
are satisfied.   If $x_{B+i} = 0$ for all $i\in A\smallsetminus B$ then by Lemma \ref{lem:vanishing}, each $w_{B+i}$ is also zero, and so the bend relations of $\sum_{i\in A \smallsetminus B} w_{A-i} w_{B+i}$ are trivially satisfied.
\end{proof}

\begin{example}
Let us return to the graphic matroid $M(K_4)$ from Example \ref{ex:graphic}.  As mentioned there, the module $Q_w$ is the quotient of the free module on $x_1,\ldots,x_6$ by the bend relations of the linear forms corresponding to the rows in the matrix depicted in that example.  By Lemma \ref{lem:vanishing}, in $\ext^3Q_w$ we have $x_I = -\infty$ for precisely the four hyperplanes (in the matroid sense) $I = \{1,2,3\},\{1,4,5\},\{2,5,6\},\{3,4,6\}$.  For instance, to see that $x_{123} \sim -\infty$ we can take the bend relations from the first row, namely $x_1+x_2 \sim x_1+ x_3 \sim x_2+x_3$, and wedge these with $x_{12}$.  On the other hand, all the variables corresponding to basis elements are identified: $x_I = x_J$ for all $I,J\in\mathcal{B}$.  For instance, wedging the bend relations of the first row with $x_{34}$ shows that $x_{134} \sim x_{234}$.  Thus $\ext^3 Q_w \cong \mathbb{B}$ is free of rank one, and it is generated by $x_I$ for any basis $I\in \mathcal{B}$ of the matroid.  
\end{example}

Generalizing the previous example, we see that for any matroid with set of bases $\mathcal{B} \subset \binom{[n]}{d}$, the module $\ext^d Q_w \cong \mathbb{B}$, where $w = \sum_{I\in\mathcal{B}}e_I$, identifies all $x_I$ for $I\in\mathcal{B}$ with the unique generator of this module and identifies $x_I$ for $I$ a matroid hyperplane with $-\infty$.

\begin{remark}
One might wonder what $\ext^d Q_w$ can look like when $w$ fails to be a Pl\"ucker vector.  From the proof, it is easy to see that if all components of $w$ are nonvanishing then $\ext^d Q_w$ will be a nontrivial quotient of a free rank 1 module.  Indeed, quotienting $\ext^d V^\vee$ by only the 2-term relations of expression \eqref{eq:key-rel} produces a free module on a single generator, and then quotienting by the additional relations with more than 2 terms from \eqref{eq:general-xw-expression} yields a nontrivial quotient of this.
\end{remark}

\subsection{Recovering the Pl\"ucker vector of a tropical linear space}

Let $L_w \subset V$ be a rank $d$ tropical linear space associated to a tropical Pl\"ucker vector $w\in\PP\left(\ext^d V\right)$.  Speyer used the perspective of matroid polytopes to prove that $w$ is uniquely determined by $L_w$, at least when $S = \T$ (and when the underlying matroid is uniform, though this latter assumption was removed in \cite{Maclagan-Sturmfels}) \cite[Proposition 2.8]{Speyer}.  We shall provide a computational method for producing $w$ from $L_w$ which more closely mimics the classical situation.  

Recall that if $L \subset k^n$ is a linear subspace of dimension $d$, then $\ext^d L \subset \ext^d k^n$ is one-dimensional, and the coefficients
expressing a generator for this one-dimensional subspace in the standard basis form
the Pl\"ucker vector for $L$.  Since our exterior algebra is defined for quotients rather than submodules, a duality is necessary to translate the top tropical wedge power into a vector.  

\begin{proposition}
  Let $L_w\subset V$ be a tropical linear space of rank $d$.  The linear dual of the quotient map
  $\ext^d V^\vee \twoheadrightarrow \ext^d Q_w$ identifies $\left( \ext^d Q_w \right)^\vee$ with the tropical Pl\"ucker vector $w \in \PP\left(\ext^d V\right)$.
\end{proposition}

\begin{proof}
We have shown in Theorem \ref{thm:main} that $\ext^dQ_w \cong S$, so it indeed dualizes to a rank one submodule of $\ext^d V$ and hence a point of the projectivization.  That this point coincides with $w$ follows immediately from Lemma \ref{lem:vanishing} and equation \eqref{eq:scalar-rel}.
\end{proof}


\section{The pairing on $\ext Q_w$}

For each rank $d$ tropical linear space, encoded by a $d$-multivector $w \in \ext^d V$, we have a graded $S$-algebra $\ext Q_w$ satisfying $\ext^k Q_w = 0$ for $k > d$ and $\ext^d Q_w \cong S$.  It is natural then to consider, for each $1 \le k \le d$, the bilinear map
\[
\ext^k Q_w \times \ext^{d-k} Q_w \to \ext^d Q_w \cong S,
\]
and in particular to ask whether this is a perfect pairing.  Felipe Rinc\'on has found a counterexample with $k=1$ to this general assertion and he has kindly allowed us to describe it here.

\begin{example}
Consider $U_{3,4}$, the uniform matroid of rank 3 on $[4] = \{1,2,3,4\}$.  This corresponds to the multivector $w = e_{123} + e_{124} + e_{134} + e_{234} \in \ext^3 V$, where $V \cong S^4$ and $S=\mathbb{B}$.  If the pairing \[\ext^2 Q_w \times \ext^1 Q_w \rightarrow \ext^3 Q_w \cong S\] were perfect then the map $\ext^2 Q_2 \rightarrow Q_w^\vee = L_w$ would be an isomorphism, but we see as follows that it is not injective.  We have 
\begin{eqnarray*}
(x_{12} + x_{34}) \wedge x_1 = x_{134}\\ 
(x_{12} + x_{34}) \wedge x_2 = x_{234}\\ 
(x_{12} + x_{34}) \wedge x_3 = x_{123}\\ 
(x_{12} + x_{34}) \wedge x_4 = x_{124}\\ 
\end{eqnarray*}
and all these $x_{ijk}$ are identified under the quotient map $\ext^3 V \twoheadrightarrow \ext^3 Q_w \cong S$, so the element $x_{12} + x_{34} \in \ext^2 Q_w$ gets mapped to $e_1 + e_2 + e_3 + e_4 \in \ext^3 Q_w$.  The same reasoning shows that $x_{13} + x_{24}$ also gets sent to $e_1+e_2+e_3+e_4$, so it suffices to show that $x_{12} + x_{34}$ and $x_{13} + x_{24}$ remain distinct under the quotient map $\ext^2 V \twoheadrightarrow \ext^2 Q_w$.  Now $Q_w = V^\vee / \bend(x_1+x_2+x_3+x_4)$ so by Lemma \ref{lem:quot} the relations in $\ext^2 Q_w$ are generated by those of the form $x_{ij} + x_{ik} =  x_{ij} + x_{il}$ and $x_{ik} + x_{jk} = x_{ik} + x_{jk} + x_{kl}$; since the only binomials in these involve pairs of indices with an index in common, they cannot not imply the relation $x_{12} + x_{34} = x_{13} + x_{24}$.
\end{example}

On the other hand, while injectivity can fail for $k=1$, surjectivity always holds.

\begin{proposition}
Assume $S$ is a totally ordered semifield. The map $\ext^{d-1} Q_w \to Q_w^\vee = L_w$ is surjective.
\end{proposition}
\begin{proof}
Since $L_w$ is spanned by the valuated cocircuit vectors of $w$ (see Theorem \ref{thm:spanned-by-cocircuits}), it suffices to show that any valuated cocircuit vector  is in the image of the composition
\[
\ext^{d-1} V^\vee \to \ext^{d-1} Q_w \to Q_w^\vee = L_w.
\]
Fix a generator $x_{K} \in \ext^d Q_w$, and recall that we have the relations $x_{J} = (w_J/w_K)x_K$.  
Given $I\in \binom{[n]}{d-1}$, this map sends $x_I$ to the element of $Q_w^\vee$ given by 
\[
x_j \mapsto x_{I+j} = \left( \frac{w_{I+j}}{w_K} \right)x_K. 
\]
In other words, $x_I$ is sent to the vector $u_I := \sum_j \left(\frac{w_{I+j}}{w_K} \right) e_j \in V$. As $I$ ranges over the $d-1$ element subsets of the ground set $[n]$, the vectors $u_I$ range over all the valuated cocircuit vectors of $w$, which span $L_w$ by Theorem \ref{thm:spanned-by-cocircuits}.
\end{proof}

An interesting open question is to give some kind of explicit description of the module of linear relations among the valuated cocircuits of a tropical Pl\"ucker vector $w$.  Since $\ext^{d-1} Q_w$ is presented as an explicit quotient of the free module with one generator for each size $d-1$ subset $K$ of $[n]$ (i.e., one generator for each cocircuit vector $\beta_K$), the question can be reformulated as asking for an explicit description of the congruence kernel of the map $\ext^{d-1} Q_w \to Q_w^\vee = L_w$. We leave this as an open question.

\bibliographystyle{amsalpha}
\bibliography{bib}

\end{document}